\documentclass[12pt,letterpaper]{amsart}

\usepackage[top=3.0cm, bottom=3.5cm, left=3.6cm, right=3.6cm]{geometry}
\newtheorem{theorem}{Theorem}[section]
\newtheorem{lemma}[theorem]{Lemma}

\usepackage{amsmath}
\usepackage{amsfonts}
\newtheorem{definition}[theorem]{Definition}
\newtheorem{example}[theorem]{Example}

\numberwithin{equation}{section}

\begin{document}

\title{Composition operators and rational inner functions on the bidisc}


\author{Athanasios Beslikas}

\begin{abstract} In the present article, composition operators induced by Rational Inner Functions on the bidisc $\mathbb{D}^2$ are studied, acting on the weighted Bergman space $A^2_{\beta}(\mathbb{D}^2).$ We prove under mild conditions that Rational Inner Functions with one singularity on $\mathbb{T}^2$ induce unbounded composition operator on $A^2(\mathbb{D}^2).$ We also prove that under the condition of stability of the polynomial inducing the Rational Inner Function, the composition operator is bounded between two different Bergman spaces.  
\end{abstract}
\footnote{MSC classification: 32A37, 32A40, 30J10,\\
The author was partially supported by SHENG-3 K/NCN/000286, research project UMO2023/Q/ST1/00048}
\maketitle
\section{Introduction} Given a Banach space $X$, a domain $\Omega\subset \mathbb{C}^n$ and a holomorphic function $f$ on $\Omega,$ the composition operator is defined as $C_{\Phi}(f)=f\circ \Phi$ for $\Phi:\Omega\to \Omega$ a holomorphic self-map of $\Omega$. In the past years, the study of the composition operators has attracted a lot of attention, after the seminal work of J.Shapiro for Hardy and Bergman spaces on the unit disc $\mathbb{D}$ (see \cite{Shapiro}). More recently, the focus has turned to other domains and settings, in general. One of the settings where the study of the boundedness, compactness and other properties of the composition operator is relatively new, is in domains of $\mathbb{C}^n.$ Several authors have provided results for the Hardy and Bergman space on the unit ball $\mathbb{B}^n$ of $\mathbb{C}^n$, or the polydisc $\mathbb{D}^n.$ To be precise, in the article of Wogen \cite{Wogen}, a characterization of the holomorphic self-maps of the unit ball (which are smooth up to the boundary) that induce bounded composition operators is given, in terms of the first derivative of the symbol $\Phi$. In the works of Bayart and Kosi\'nski (see \cite{Bayart2},\cite{Kosinski}), a full characterization of the holomorphic self-maps of the bidisc (again requiring some smoothness up to the distinguished boundary $\mathbb{T}^2$) that induce bounded composition operator acting on the weighted Bergman space $A^2_{\beta}(\mathbb{D}^2)$ is provided:
\begin{theorem} Let $\Phi \in \mathcal{O}(\mathbb{D}^2,\mathbb{D}^2)\cap \mathcal{C}^2(\overline{\mathbb{D}^2})$. Then the composition operator $C_{\Phi}:A^2_{\beta}(\mathbb{D}^2)\to A^2_{\beta}(\mathbb{D}^2) $ is bounded, if and only if for all $\zeta\in \mathbb{T}^2$ such that $\Phi(\zeta)\in \mathbb{T}^2,$ the derivative $d_{\zeta} \Phi$ is invertible. 
\label{characterization}
\end{theorem}
In the work of Kosi\'nski \cite{Kosinski} the condition "the derivative $d_{\zeta} \Phi$ is invertible for all $\zeta\in \mathbb{T}^2$ such that $\Phi(\zeta)\in \mathbb{T}^2$" is called "first-order condition" because it refers only to the first derivative of the symbol $\Phi.$
Both of the articles \cite{Bayart2}, \cite{Kosinski} are quite elegant and technical but only take care of the case where the symbol $\Phi:\mathbb{D}^2\to \mathbb{D}^2$ has coordinate functions that have some type of smoothness on the distinguished boundary of the bidisc $\mathbb{D}^2.$ In the present work, the case where the symbol has some type of singularity on the boundary is considered.  Some interesting results arise when we consider the coordinate functions to be Rational Inner Functions defined on the bidisc.\\ Rational Inner Functions and their properties are currently studied, and several authors have obtained a substantial amount of very interesting results. In the works of Sola, Knese, Bickel, Pascoe, $\mathrm{M}^c$Carthy, Bergqvist and others (see \cite{Sola}, \cite{Knese1}, \cite{Bickel}, \cite{Pascoe}, \cite{Bergvist}, \cite{McCarthy}) one can find results about the regularity and smoothness of Rational Inner Functions (or RIFs) and membership of RIFs in holomorphic function spaces (Hardy, Dirichlet spaces).\\ In the present work part of the mentioned Theorems will be applied to obtain results regarding the boundedness of the composition operator when the symbol $\Phi$ has RIFs as coordinate functions, applying the tools that Kosi\'nski and Bayart did in their articles \cite{Kosinski}, \cite{Bayart2}. The main goal of this note is to answer the following question.\\\\
\textbf{Question 1:} \textit{Assume that $\Phi:\mathbb{D}^2\to\mathbb{D}^2,$ where $\Phi=(\varphi_1,\varphi_2)$ is a holomorphic self map of the bidisc and at least one of the coordinate functions $\varphi_i, i=1,2$ is a Rational Inner Function on the bidisc, hence not always smooth on the distinguished boundary $\mathbb{T}^2.$ What can one say about the boundedness of the composition operator $C_{\Phi}:A^2_{\beta}(\mathbb{D}^2)\to A^2_{\beta}(\mathbb{D}^2)$ induced by such symbols?}
\subsection{Notation} The standard notation used in the literature will be followed throughout the whole paper: $\mathbb{D}^2$ denotes the unit bidisc, that is
$$\mathbb{D}^2=\{z=(z_1,z_2)\in\mathbb{C}^2:|z_1|<1,|z_2|<1\}.$$
By $\mathbb{T}^2$ we denote the distinguished boundary of the bidisc, the bitorus, which is defined as 
$$\mathbb{T}^2=\{\zeta=(\zeta_1,\zeta_2)\in\mathbb{C}^2:|\zeta_1|=1,|\zeta_2|=1\}.$$
By $\mathcal{O}(\Omega_1,\Omega_2)$ we denote the class of holomorphic functions $f:\Omega_1\to \Omega_2, \Omega_1\subset\mathbb{C}^n,\Omega_2\subset \mathbb{C}^m.$ Positive constants in estimates will be denoted by $C,$ or $C'.$ Other notations that will appear in the paper will be explained on the spot. 
\subsection{Organization of the paper} In Section 2 the main results of the paper are discussed. Section 3 is divided into three shorter subsections. In Subsection 3.1, background for the weighted Bergman and spaces is given and in Subsection 3.2. some introductory information regarding Rational Inner Functions (RIFs) is provided. In Section 3.3 the main tools are listed. Sections 4 and 5 are dedicated to the proofs of the results and some examples, respectively. 
\subsection{Acknowledgements} I would like to thank my supervisor, \L{}ukasz Kosi\'nski for discussing with me the main idea of this note and my friends and colleagues Marcin Tombi\'nski, Anand Chavan and Pouriya Torkinejad for the valuable discussions. Most importantly, I express my gratitude to professor Alan Sola, for his remarks, comments and general correspondence during the preparation of this work. I also express my sincere gratitude for the anonymous referee, for his careful reading of the paper, and the constructive suggestions that improved the shape of the paper.
\section{Main results} In this Section, a brief overview of the results of this work is given. The main result that we focus on is \ref{th2}. In this result we prove that if a RIF has one singularity on $\mathbb{T}^2$ and the polynomial in the denominator satisfies a certain condition, then the composition operator $C_{\Phi}:A^2(\mathbb{D}^2)\to A^2(\mathbb{D}^2)$ is not bounded. In contrast, we prove that the operator $C_{\Phi}:A^2_{\frac{\beta}{2}-2}(\mathbb{D}^2)\to A^2_{\beta}(\mathbb{D}^2)$ is bounded, whenever the RIFs are induced by stable polynomials. The third result is a volume condition for symbols that are induced by RIFs in the Schur-Agler class in the polydisc. Let us now proceed with the precise statements of the results.
\begin{theorem} Let $\Phi\in\mathcal{O}(\mathbb{D}^2,\mathbb{D}^2),$ $\Phi=(\varphi,\varphi)$ where $\varphi$ is a RIF.  Assume that $\varphi(z_1,z_2)=\frac{\tilde{p}(z_1,z_2)}{p(z_1,z_2)}$ where $p\in \mathbb{C}[z_1,z_2]$ is a polynomial which is stable on $\mathbb{D}^2,$ with a single zero on $\mathbb{T}^2.$ If there exists an exponent $q>1$ and a neighborhood $\mathcal{U}$ of $\tau$ such that
$$|p(z)|\ge C\mathrm{dist}^q(z,\mathcal{Z}_p),z\in\overline{\mathbb{D}^2\cap \mathcal{U}}$$
  then, $C_{\Phi}: A^2(\mathbb{D}^2)\to A^2(\mathbb{D}^2)$ is not bounded.
  \label{th2}
\end{theorem}
The next Theorem considers symbols that are induced by the same RIF as co-ordinate function, which, in turn, is induced by stable polynomials on $\overline{\mathbb{D}^2}.$
\begin{theorem} Let $\Phi=(\varphi,\varphi)\in \mathcal{O}(\mathbb{D}^2,\mathbb{D}^2), \varphi=\frac{\tilde{p}}{p},$ $p$ a stable polynomial on $\overline{\mathbb{D}^2}.$ Then, for all $\beta>4,$ the composition operator
$C_{\Phi}:A^2_{\frac{\beta}{2}-2}(\mathbb{D}^2)\to A^2_{\beta}(\mathbb{D}^2)$ is bounded.
\label{th3}
\end{theorem}
The last result is a sufficient volume condition with the assumption that  all coordinate functions belong to the \textit{Schur-Agler class} of RIFs defined on the polydisc $\mathbb{D}^n$.
\begin{theorem}
Let $\Phi=(\varphi_1,...,\varphi_n):\mathbb{D}^n\to\mathbb{D}^n$, where $\varphi_i=\frac{\tilde{p_i}}{p_i},i=1,...,n$ rational inner functions on the bidisc that belong to the Schur-Agler class. If the $V_\beta$-volumes satisfy
$$V_{\beta}\left(\biggl\{z\in \mathbb{D}^n:\sum_{j=1}^n(1-|z_j|^2)SOS_{ij}\leq \delta_i M\biggr\}\right)\leq C\delta_i^{n(\beta+2)}$$
for $i=1,...,n,$ for all $\delta_i\in(0,1),$ for some constant $M>0,$ then, the composition operator $C_{\Phi}: A_{\beta}^2(\mathbb{D}^n)\to A^2_{\beta}(\mathbb{D}^n)$
is bounded, where $\beta\ge -1,$ and $SOS_{ij}$ are sums of squared moduli of  the polynomials $p_i\in \mathbb{C}[z_1,...,z_n].$
\label{th4}
\end{theorem}
A brief mention for the Schur-Agler class of holomorphic functions on the polydisc is made in Section 3, as well as the sum of squares formula for polynomials $p\in\mathbb{C}[z_1,z_2].$
\section{Background material and tools.}
\subsection{Background on function spaces} As stated in the introduction, we focus on the weighted Bergman spaces on the bidisc. The weighted Bergman spaces $A_{\beta}^2(\mathbb{D}^2)$ on the bidisc are consisted of the holomorphic functions $f\in \mathcal{O}(\mathbb{D}^2,\mathbb{C})$ such that
$$\int_{\mathbb{D}^2}|f(z_1,z_2)|^2(1-|z_1|^2)^{\beta}(1-|z_2|^2)^{\beta}dA(z_1)dA(z_2)<\infty,$$ where $\beta\ge -1$ and $dA(z)=\frac{1}{\pi}dxdy,$ is the normalised Lebesgue measure on the disc. These spaces are of course Hilbert spaces, and enjoy the property of bounded point evaluation functional.\\
A lot of work has been done regarding the membership of RIFs into the Hardy and Dirichlet-type spaces on the bidisc, but not in the Bergman spaces. We remind the interested reader that the anisotropic Dirichlet-type spaces on the bidisc $\mathfrak{D}_{\vec{a}}(\mathbb{D}^2)$ were defined by Kaptanoglu (see \cite{Kaptanoglu} for more information about the Dirichlet-type spaces on polydiscs.)
\subsection{Backround on RIFs on the bidisc.}
The study of Rational Inner Functions (RIFs) has generated important results recently. In the one-dimensional setting, a holomorphic function $f$ on $\mathbb{D}$ is called \textit{inner} if $|f|=1$ a.e. on the unit circle $\mathbb{T}.$ In the one dimensional setting the rational functions that are \textit{inner} are the finite Blaschke products. In the multivariate setting, in order to characterize the RIFs we need polynomials in two variables. For simplicity, we present elements of this theory on two variables, but the results hold for higher complex dimension. Consider a polynomial $p\in \mathbb{C}[z_1,z_2]$ with bidegree $(n,m)\in \mathbb{N}^2.$ The reflection of a polynomial $p\in\mathbb{C}[z_1,z_2]$ is defined as 
$$\tilde{p}(z)=z_1^nz_2^m\overline{p\left(\frac{1}{\overline{z_1}},\frac{1}{\overline{z_2}}\right)}$$
Rudin's theorem (see \cite{Rudin}) assures that every RIF on the bidisc setting (or polydisc if modified accordingly), has the following representation:
$$\varphi(z_1,z_2)=\lambda z_1^Mz_2^N\frac{\tilde{p}(z_1,z_2)}{p(z_1,z_2)},$$ 

where $\lambda\in \mathbb{C}$ is a unimodular constant.
Contrary to what happens in the one dimension, in 2 or more complex dimensions RIFs may have singularities. Let $p\in\mathbb{C}[z_1,z_2]$ which has no zeros on the open bidisc $\mathbb{D}^2$ and has zeros only on the distinguished boundary $\mathbb{T}^2.$ Let $\varphi(z_1,z_2)=\lambda\frac{\tilde{p}(z_1,z_2)}{p(z_1,z_2)}$ be the corresponding RIF. If a point $\zeta \in \mathbb{T}^2$ belongs to the zero set $\mathcal{Z}_p$ then we say that it is a \textit{singular point} of the RIF $\varphi.$ Even though RIFs in two dimensions have singularities, these singularities can be tamed in the sense that the non-tangential limit exists always, even on the singular points (see Theorem C of \cite{Knese2}) and is an element of $\mathbb{T}.$ Higher notions of regularity also exist for these functions on their singular points (see for example \cite{Knese1}, \cite{Pascoe}, \cite{Bickel} and \cite{Knese2}). 
Last but not least, a brief mention has to be made regarding RIFs in the \textit{Schur-Agler class}. We provide the definition on the polydisc setting:
\begin{definition}
If $f :\mathbb{D}^n\to \mathbb{D}$ is holomorphic, then we say $f$ satisfies
the von Neumann inequality or $f$ is in the Schur-Agler class if
$$||f(T)|| \leq 1$$
for all commuting n-tuples of strict contractions $T = (T_1, . . . , T_n).$ 
\end{definition}
We remind the reader that an operator $T$ is a strict contraction, if the operator norm satisfies $||T||< 1.$
Note that all RIFs in two complex dimensions belong to the \textit{Schur-Agler class}.
\subsection{Tools} In this subsection the tools that will be used are presented. Most of the requiered Lemmata and Theorems are contained in the works of Bayart, Koo, Stessin, Zhu, and Kosi\'nski in their respective articles \cite{Bayart2}, \cite{Koo}, \cite{Kosinski}. 
The main tools used to obtain Theorem \ref{characterization}, are volume estimates based on Carleson conditions for the pull-back measure. Recall that a two-dimensional Carleson box is defined as $$S(\zeta,\tilde{\delta})=\{(z_1,z_2)\in\mathbb{D}^2:|z_1-\zeta_1|<\delta_1,|z_2-\zeta_2|<\delta_2\},$$
where
$\zeta=(\zeta_1,\zeta_2)\in \mathbb{T}^2, \tilde{\delta}=(\delta_1,\delta_2)\in(0,2)^2.$
\begin{lemma}
Let $\Phi:\mathbb{D}^2\to \mathbb{D}^2$ be a holomorphic self-map of the bidisc. The composition operator $C_{\Phi}: A^2_{\beta}(\mathbb{D}^2)\to A^2_{\beta}(\mathbb{D}^2)$ is bounded if and only if there is a constant $C_{\beta}>0$ such that for every $\tilde{\delta} \in (0,2)^2$ and $\zeta\in\mathbb{T}^2$:
$$V_{\beta}(\Phi^{-1}(S(\zeta,\tilde{\delta})))\leq C_{\beta}V_{\beta}(S(\zeta,\tilde{\delta})).$$
Moreover, the constant $C_\beta$ can be chosen independently of $\beta$ whenever $\beta \in (-1,0].$
\end{lemma}
The above Lemma, combined with the estimates that appear in the articles of Koo, Stessin, Zhu, Bayart and Kosi\'nski, generates some sufficient and necessary volume condition for boundedness of $C_{\Phi}$. To be precise, setting $S_\delta=\{z\in\mathbb{D}:|1-z|<\delta\},$ it is proved that $C^{-1}\delta^{2+\beta}\leq V_\beta(S_{\delta})\leq C\delta^{\beta+2}.$ It is easy for one to check that if we set $\Phi=(\varphi,\varphi):\mathbb{D}^2\to\mathbb{D}^2,$ and consider $\tilde{\delta}=(\delta,\delta),$ $\delta\in(0,2),$ then $$V_{\beta}(\Phi^{-1}(S(\zeta,\tilde{\delta})))=V_{\beta}(\{z\in\mathbb{D}^2:|\varphi(z)-\zeta|<\delta\}).$$
If one now proves that for some $C>0$
$$V_{\beta}(\{z\in\mathbb{D}^2:|\varphi(z)-\zeta|<\delta\})>C\delta^{2a+4}$$
for some $\zeta\in\mathbb{T}$ and for $\delta\in(0,2),$ then the operator $C_{\Phi}:A^2_{a}(\mathbb{D}^2)\to A^2_{\beta}(\mathbb{D}^2)$ is not bounded. In contrast, if one proves the reverse inequality for all $\zeta \in \mathbb{T}$ and for all $\tilde{\delta}=(\delta_1,\delta_2)\in(0,2)^2,$ then the boundedness of $C_{\Phi}:A^2_{a}(\mathbb{D}^2)\to A^2_{\beta}(\mathbb{D}^2)$ is proved. 
This observation will be crucial in the sequel.
A well established Theorem that will turn out to be proved pivotal for the proof of Theorem \ref{th2}, is the classical \L{}ojasiewicz inequality. More information on this inequality can be found in the work \cite{Loja}. An application of \L{}ojasiewicz inequality has been made in \cite{Kosinski2} and \cite{Vavitsas} to prove cyclicity conditions for polynomials in Dirichlet-type spaces on the bidisc and on the two-dimensional unit ball.
\begin{theorem} (\L{}ojasiewicz inequality) Let  $f: U \to  \mathbb{R}$ be a real analytic function on an open set $U$ in $\mathbb{R}^n$, and let $\mathcal{Z}_f$ be the zero set of $f$. Assume that $\mathcal{Z}_f$ is not empty. Then, for any compact set $K$ in $U$, there exist positive constants $q$ and $C$ such that, for all $x$ in $K$
$$\mathrm{dist}^q(x,\mathcal{Z}_f)\leq C|f(x)|.$$
\end{theorem}
Here "$\mathrm{dist}(x,\mathcal{Z}_f)$" denotes the Euclidean distance of a point $x\in K$ from the zero set of the function $f.$ 
The next result is related to RIFs that belong to the \textit{Schur-Agler} class. 
\begin{theorem} Given a polynomial $p \in \mathbb{C}[z_1,...,z_n]$ with no zeros in $\mathbb{D}^n$ and poly-degree at most \textbf{m}$=(m,...,m)$. The function $\tilde{p}/p$ is in the Schur-Agler class exactly when
$$|p(z)|^2-|\tilde{p}(z)|^2=\sum_{j=1}^n(1-|z_j|^2)SOS_j(z),$$
where $SOS_j$ denotes the sum of squared moduli of polynomials.
\end{theorem}
An example of a specific polynomial and its Agler decomposition follows.
\begin{example} \textit{Consider the polynomial $p(z)=2-z_1-z_2,$ where $(z_1,z_2)\in\overline{\mathbb{D}^2}.$ Then $\tilde{p}(z)=2z_1z_2-z_1-z_2$ and}
$$|\tilde{p}(z)|^2-|p(z)|^2=2(1-|z_1|^2)|1-z_2|^2+2(1-|z_2|^2)|1-z_1|^2,(z_1,z_2)\in \overline{\mathbb{D}^2}.$$
\label{easyexample}
\end{example}
The next result appears in \cite{Bickel0} and considers polynomials which are stable in $\overline{\mathbb{D}^2}.$ We remind the reader that a polynomial is called stable on $\overline{\mathbb{D}^2}$ if it has no zeros on $\overline{\mathbb{D}^2}.$ One easy example of such a polynomial is $p(z)=3-z_1-z_2,$ where $(z_1,z_2)\in \overline{\mathbb{D}^2}$
\begin{theorem} Let $p\in \mathbb{C}[z_1,z_2]$. Then $p$ is stable on $\overline{\mathbb{D}^2}$ if and only if there exists a constant $C>0$ such that
$$|\tilde{p}(z)|^2-|p(z)|^2\ge C(1-|z_1|^2)(1-|z_2|^2).$$
\label{bickelstable}
\end{theorem}
For more details about the Schur-Agler class the reader can study, for instance, the article of Knese \cite{Knese1}, and the references therein. Example \ref{easyexample} features in \cite{Knese2} as well.

\section{Proofs}
This section is dedicated to the proofs of the main results stated before. We begin with the proof of Theorem \ref{th2}. For convenience, we will provide again the statement of the results before their proofs.\\\\
\textbf{Theorem \ref{th2}}.\textit{Let $\Phi\in\mathcal{O}(\mathbb{D}^2,\mathbb{D}^2),$ $\Phi=(\varphi,\varphi)$ where $\varphi$ is a RIF.  Assume that $\varphi(z_1,z_2)=\frac{\tilde{p}(z_1,z_2)}{p(z_1,z_2)}$ where $p\in \mathbb{C}[z_1,z_2]$ is a polynomial which is stable on $\mathbb{D}^2,$ with a single zero on $\mathbb{T}^2.$ If there exists an exponent $q>1$ and a neighborhood $\mathcal{U}$ of $\tau$ such that
$$|p(z)|\ge C\mathrm{dist}^q(z,\mathcal{Z}_p),z\in\overline{\mathbb{D}^2\cap \mathcal{U}}$$
  then, $C_{\Phi}: A^2(\mathbb{D}^2)\to A^2(\mathbb{D}^2)$ is not bounded.}

\begin{proof}
Set $\Phi=(\varphi,\varphi):\mathbb{D}^2\to \mathbb{D}^2,$ following the assumption that $\varphi(z_1,z_2)=\frac{\tilde{p}(z_1,z_2)}{p(z_1,z_2)}.$ Assume, without losing the generality, that $\varphi$ has non tangential value equal to 1 at the point $\tau=(1,1).$ As in \cite{Bayart2} and \cite{Kosinski} we are interested in estimating the volume $V_{\beta}(\Phi^{-1}(S(e,\tilde{\delta}))), \tau=(1,1),$ for $\tilde{\delta}=(\delta_1,\delta_1)$ which, by our choice of $\Phi,$ is equivalent to estimate the volume
$$V_{\beta}(\{z\in \mathbb{D}^2:|\varphi(z)-1|\leq \delta_1\}),\delta_1\in(0,1).$$
Let us draw a sketch of the proof. The idea is to bound from above the difference $|\varphi(z)-1|.$ This will allow us to find subsets of the sublevel sets $\{z\in \mathbb{D}^2:|\varphi(z)-1|\leq \delta_1\}$ which will have easier volume to calculate.
Let us postpone the proof of the Theorem \ref{th2}. The following lemma is crucial. 
\begin{lemma} Let $\varphi$ be a RIF which has only one singularity on $\tau=(\tau_1,\tau_2) \in \mathbb{T}^2$ and has non tangential value equal to 1 at $\tau\in\mathbb{T}^2.$ Then there exists  $\epsilon,q>0$, a constant $C>0$ and a neighborhood $\mathcal{U}$ of $\tau,$ such that
$$|\varphi(z)-1|\leq \frac{C\epsilon}{(|z_1-\tau_1|^2+|z_2-\tau_2|^2)^{\frac{q}{2}}},z=(z_1,z_2)\in \overline{\mathbb{D}^2\cap \mathcal{U}}. $$
\end{lemma}
\begin{proof}
Without losing the generality, assume that $\tau=(1,1).$
The idea is to apply \L{}ojasiewicz inequality for a compact set, and find a local upper bound for the difference $|\varphi(z)-1|$. Consider $D(\tau,\tilde{r}),$ where $\tilde{r}=(r,r)$ for $r>0$ a closed bidisc around $\tau.$ The set $\overline{\mathbb{D}^2}\cap \overline{D(\tau,\tilde{r})}$ is compact. Applying a continuity argument for  $|\tilde{p}(z)-p(z)|=|\tilde{p}(z)-p(z)-\tilde{p}(\tau)+p(\tau)|$ near $\tau \in \mathbb{T}^2$ and combining it with  \L{}ojasiewicz inequality on the compact set $\overline{\mathbb{D}^2}\cap \overline{D(\tau,\tilde{r})}$ one can find  $q,\epsilon>0$ and a positive constant $C>0,$ such that the following chain holds:
\begin{align}
|\varphi(z)-1|=&\left|\frac{\tilde{p}(z)}{p(z)}-1\right|&&\\
=&\left|\frac{\tilde{p}(z)-p(z)}{p(z)}\right|&&\\
\leq& \frac{\epsilon}{|p(z)|}&&\\ \leq
&\frac{\epsilon}{C\mathrm{dist}^q(z,\mathcal{Z}_{p}\cap \mathbb{T}^2)}&&\\
= &\frac{\epsilon}{C(|z_1-1|^2+|z_2-1|^2)^{\frac{q}{2}}}, 
\end{align}
In particular, inequality (4.3) is obtained using the above mentioned continuity argument, and inequality on (4.4) was obtained after applying \textit{\L{}ojasiewicz inequality} for the polynomial $p(z)$ on the compact set $\overline{\mathbb{D}^2}\cap \overline{D(\tau,\tilde{r})}.$
\end{proof}
At this point we can proceed to the proof of the non-boundedness Theorem.
Fix an $r>0$ such that the bi-annulus $\left(\frac{\epsilon}{2\delta_1}\right)^{\frac{1}{q}}<|z_i-1|<\epsilon^{\frac{1}{q}}$ where $i=1,2$ is contained into $\overline{\mathbb{D}^2\cap D(\tau,\tilde{r})},$  for all $1>\delta_1>1/2.$ Essentially we "cut" a small bi-annulus near the singularity, contained inside $\overline{\mathbb{D}^2\cap D(\tau,\tilde{r})}.$  After some calculations one observes that 
if $z_1,z_2$ reside in the biannulus mentioned above, then $\frac{\epsilon}{(|z_1-1|^2+|z_2-1|^2)^{\frac{q}{2}}}<\delta_1$ which implies, by Lemma 4.3, that $|\varphi(z)-1|<\delta_1.$ Now let $\tau=(1,1),$ $\tilde{\delta}=(\delta_1,\delta_1).$ For $\epsilon>0$ small enough, the volume of the inverse image under $\Phi$ of the Carleson box will contain a section of the above mentioned bi-annulus. This implies that there exist a positive constant $1>C>0$ such that the upcoming chain holds. 
\begin{align}
V(\Phi^{-1}(S(e,\tilde{\delta})))=&V(\{z\in\mathbb{D}^2:|\varphi(z)-1|\leq\delta_1\})&&\\
>&V(\{z\in \mathbb{D}^2\cap \mathcal{U}:|\varphi(z)-1|\leq \delta_1\})&&\\
>&\int_{\{z\in\mathbb{D}^2:\left(\frac{\epsilon}{2\delta_1}\right)^{\frac{1}{q}}<|z_i-1|<\epsilon^{\frac{1}{q}},i=1,2\}}1dV(z_1,z_2)&&\\
>&C\pi^2 \left( \epsilon^{\frac{2}{q}}-\left(\frac{\epsilon}{2\delta_1}\right)^{\frac{2}{q}} \right)^2&&\\
>&C'(\epsilon)\delta_1^{4/q}&&\\
>&C'(\epsilon)\delta_1^{4}
\end{align}for all $\delta_1\in I,$ where $C(\epsilon)>0$ is a constant depending only on $\epsilon,$ and $\epsilon$ is independent of $\delta_1.$ Following our assumption that $q>1,$ and since $\epsilon$ is a small fixed positive real number, the above arguments prevent $C_{\Phi}$ to be bounded on $A^2(\mathbb{D}^2).$
\end{proof}
\textbf{Theorem \ref{th3}} \textit{Let $\Phi=(\varphi,\varphi)\in \mathcal{O}(\mathbb{D}^2,\mathbb{D}^2), \varphi=\frac{\tilde{p}}{p},$ $p$ a stable polynomial on $\overline{\mathbb{D}^2}.$ Then, for all $\beta>4,$ the composition operator
$C_{\Phi}:A^2_{\frac{\beta}{2}-2}(\mathbb{D}^2)\to A^2_{\beta}(\mathbb{D}^2)$ is bounded.}

\begin{proof}
 (\textit{Proof of Theorem \ref{th3}.})
 We have to estimate the $V_{\beta}-$ volumes for every $\zeta \in \mathbb{T}, $ for all $\delta_i \in (0,2),i=1,2$ of the sublevel set $\{z\in \mathbb{D}^2:|\varphi(z)-\zeta|<\delta_i\},$ where $\varphi=\frac{\tilde{p}}{p}$ a RIF which is induced by a polynomial which is stable on $\overline{\mathbb{D}^2}.$ For all $\zeta \in \mathbb{T},$ by an application of Theorem \ref{bickelstable} and the Maximum Modulus Principle
\begin{align}
|\varphi(z)-\zeta| =&\left|\frac{\tilde{p}(z)}{p(z)}-\zeta\right|&&\\
                   \ge&\left|\frac{|\tilde{p}(z)|-|p(z)|}{p(z)}\right|&&\\
                   \ge&\frac{C(1-|z_1|^2)(1-|z_2|^2)}{2M^2},            
\end{align}
Here, $M=\max_{z\in\overline{\mathbb{D}^2}}|p(z)|.$
The above facts imply that the sublevel sets $\{z\in \mathbb{D}^2:|\varphi(z)-\zeta|<\delta_i\},$ are contained in the sublevel sets $\{z\in\mathbb{D}^2:(1-|z_1|^2)(1-|z_2|^2)\leq 2M^2\delta_i\}$ for all $\delta_i\in(0,2).$
Hence,
\begin{align}
V_{\beta}(\Phi^{-1}(S(\zeta,\delta)))\leq& 
V_{\beta}(\{z\in \mathbb{D}^2:|\varphi(z)-\zeta|<\delta_1\})&&\\
\leq & V_{\beta}(\{z\in\mathbb{D}^2:(1-|z_1|^2)(1-|z_2|^2)\leq 2M^2\delta_1\})&&\\
\leq &\int_{\{z\in\mathbb{D}^2:(1-|z_1|^2)(1-|z_2|^2)\leq 2M^2\delta_1\}}dV_{\beta}(z_1,z_2)&&\\
\leq & (2M^2)^{\beta} \delta_1^{\beta-2+2}.
\end{align}
This suffices to prove the boundedness of $C_{\Phi}:A^2_{\frac{\beta}{2}-2}(\mathbb{D}^2)\to A^2_{\beta}(\mathbb{D}^2)$ for exponents $\beta>4$
\end{proof}
 Let us briefly comment on the previous result. Using the generalized version of Theorem \ref{bickelstable} for arbitrary dimension, setting $\Phi=(\varphi,...,\varphi)\in \mathcal{O}(\mathbb{D}^n,\mathbb{D}^n),$ where $ \varphi=\frac{\tilde{p}}{p}$ a RIF induced by a stable polynomial on $\overline{\mathbb{D}^n},$ and repeating the same estimates, one receives $C_{\Phi}:A^2_{\frac{\beta}{n}-2}(\mathbb{D}^n)\to A^2_{\beta}(\mathbb{D}^n),$ for all $n\ge 2,$ and for all $\beta> 2n.$\\\\

\textbf{Theorem \ref{th4}}
\textit{
Let $\Phi=(\varphi_1,...,\varphi_n):\mathbb{D}^n\to\mathbb{D}^n$, where $\varphi_i=\frac{\tilde{p_i}}{p_i},i=1,...,n$ rational inner functions on the bidisc that belong to the Schur-Agler class. If the $V_\beta$-volumes satisfy
$$V_{\beta}\left(\biggl\{z\in \mathbb{D}^n:\sum_{j=1}^n(1-|z_j|^2)SOS_{ij}\leq \delta_i M\biggr\}\right)\leq C\delta_i^{n(\beta+2)}$$
for $i=1,...,n,$ for all $\delta_i\in(0,1),$ for some constant $M>0,$ then, the composition operator $C_{\Phi}: A_{\beta}^2(\mathbb{D}^n)\to A^2_{\beta}(\mathbb{D}^n)$
is bounded, where $\beta\ge -1,$ and $SOS_{ij}$ are sums of squared moduli of  the polynomials $p_i\in \mathbb{C}[z_1,...,z_n].$}
\begin{proof}
Set $\Phi=(\varphi_1,...,\varphi_n):\mathbb{D}^n\to \mathbb{D}^n$ where  $\varphi_i,i=1,2$ are RIFs that belong to the Schur-Agler class on the polydisc. Initially, we have:
$$|p_i(z)|^2-|\tilde{p}_i(z)|^2=\sum_{j=1}^{n}(1-|z_j|^2)SOS_{ij}(z),$$
which comes from the Agler decomposition of the polynomials $p_i.$
Developing the difference of squares on the left hand side, gives us
$$2M_i(|p_i(z)|-|\tilde{p}_i(z)|)\ge |p_i(z)|^2-|\tilde{p}_i(z)|^2=\sum_{j=1}^{n}(1-|z_j|^2)SOS_{ij}(z),$$
where $M_i$ are the maximums of each polynomial $p_i$ on $\overline{\mathbb{D}^n}.$
At this point, we will estimate the volume 
$$V_{\beta}(\{z\in \mathbb{D}^n:|\varphi_i(z)-\zeta|\leq \delta_i\}),$$ 
for all $\zeta\in\mathbb{T}$
and for all $\delta_i\in(0,1),i=1,2...,n.$ By the triangle inequality one has:
\begin{align}
\delta_i\ge|\varphi_i(z)-\zeta|\ge&\left|\frac{|p_i(z)|-|\tilde{p}_i(z)|}{|p_i(z)|}\right|&&\\
\ge &\frac{1}{2M_i}\frac{\sum_{j=1}^{n}(1-|z_j|^2)SOS_{ij}(z)}{|p_i(z)|}&&\\
\ge &\frac{1}{2M_i^2}\sum_{j=1}^{n}(1-|z_j|^2)SOS_{ij}(z)
\end{align}
The last inequality implies that the family of sublevel sets 
$$A_{p_i}=\biggl\{z\in\mathbb{D}^n: \sum_{j=1}^{n}(1-|z_j|^2)SOS_{ij}(z) \leq 2M_i^2\delta_i \biggr\}$$ contains the sublevel sets
$S_{\varphi_i}=\{z\in\mathbb{D}^n: |\varphi_i(z)-\zeta|\leq \delta_i\},$
for all $\zeta\in \mathbb{T}.$ Therefore, following the assumption in the statement of the Theorem, one obtains that
\begin{align}
V_{\beta}(\Phi^{-1}(S(\zeta,\tilde{\delta})))\leq &V_{\beta}(\{z\in \mathbb{D}^n:|\varphi_i(z)-1|\leq \delta_i\})&&\\
\leq& C\delta_i^{n(\beta+2)}
\end{align}
\end{proof}
 Due to the absence of a Taylor expansion for a RIF near a singularity, it seems that the only viable way to estimate the volume of the sets $\Phi^{-1}(S(\zeta,\tilde{\delta}))$ is via the difference of squares formula of the corresponding polynomials. The above stated result is very difficult to verify in higher dimensions, nevertheless, in the next Section we provide an example in which we apply the above technique to obtain estimates which entail boundedness of the composition operator  $C_{\Phi}:A^2_{\frac{\beta}{4}-2}(\mathbb{D}^2)\to A^2_{\beta}(\mathbb{D}^2), \beta \ge 8,$ where the symbol is induced by a RIF with a singularity.
\section{Examples}
In this section, we consider three examples. The first example showcases boundedness of the composition operator $C_{\Phi}:A^2_{\frac{\beta}{4}-2}(\mathbb{D}^2)\to A^2_{\beta}(\mathbb{D}^2), \beta \ge 8.$ This specific example is interesting as it utilizes the difference of squares formula for the chosen $p\in\mathbb{C}[z_1,z_2]$. Using (roughly )the same symbol $\Phi$ as in Example 5.1, we will obtain non-boundedness of $C_{\Phi}$ acting on the classical Bergman space $A^2(\mathbb{D}^2).$ The approach of the proof of Theorem \ref{th2} will be followed. The comparison of the first two examples is interesting, as we observe a jump of length 8 in the switch from boundedness to non-boundedness. The reason behind this jump is probably laying into the fact that our estimates seem to not be sharp. We think that this jump can be improved. Lastly, a short example of a symbol induced by a RIF which, in turn, is induced by a stable polynomial is given, in light of Theorem \ref{th3}. 
\begin{example} Let
$$\Phi(z_1,z_2)=(\varphi,\varphi)=\left(\frac{2z_1z_2-z_1-z_2}{2-z_1-z_2},\frac{2z_1z_2-z_1-z_2}{2-z_1-z_2}\right), z_1,z_2\in \mathbb{D}.$$
Then, the composition operator $C_{\Phi}:A^2_{\frac{\beta}{4}-2}(\mathbb{D}^2)\to A^2_{\beta}(\mathbb{D}^2), \beta \ge 8$ is bounded.
\end{example}
\begin{proof}
Consider the polynomial $p(z)=2-z_1-z_2, z_1,z_2\in \mathbb{D}.$  
The Agler decomposition of the difference of squares, by Example \ref{easyexample} is 
$$|\tilde p(z)|^2-|p(z)|^2=2(1-|z_1|^2)|1-z_2|^2+2(1-|z_2|^2)|1-z_1|^2.$$ For all $\zeta \in \mathbb{T}$ one has:
\begin{align}|\varphi(z)-\zeta|=&\frac{|\tilde{p}(z)-\zeta p(z)|}{|p(z)|}&&\\
                             \ge&\frac{||\tilde{p}(z)|-|p(z)||}{|p(z)|}&&\\
                             \ge&\frac{2(1-|z_1|^2)|1-z_2|^2+2(1-|z_2|^2)|1-z_1|^2}{|p(z)|(|\tilde{p}(z)|+|p(z)|)}&&\\        \ge&\frac{1}{8}(1-|z_1|)^2(1-|z_2|)^2    
\end{align}
At this moment, take $\delta_i\in(0,2).$
\begin{align} V_{\beta}(\{z\in\mathbb{D}^2:|\varphi(z)-\zeta|<\delta_i\})\leq & \int_{\{z\in \mathbb{D}^2:(1-|z_1|)(1-|z_2|)\leq 2\sqrt{2}\delta_i^{1/2}\}}dV_{\beta}(z_1,z_2) &&\\
\leq& (2\sqrt{2})^\beta \delta_i^{\frac{\beta}{2}-2+2}
\end{align}
and the proof is complete.
\end{proof}
\begin{example} Let
$$\Phi(z_1,z_2)=(\varphi,\varphi)=\left(-\frac{2z_1z_2-z_1-z_2}{2-z_1-z_2},-\frac{2z_1z_2-z_1-z_2}{2-z_1-z_2}\right), z_1,z_2\in \mathbb{D}.$$
Then, the composition operator $C_{\Phi}:A^2(\mathbb{D}^2)\to A^2(\mathbb{D}^2)$ is not bounded.
\end{example}
\begin{proof}
The non-tangential value of  $\varphi$ at the singular point $(1,1)$ is equal to 1. Hence, one has to estimate the volume
$$V(\{z\in \mathbb{D}^2:|\varphi(z)-1|<\delta\})>V(\{z\in\mathbb{D}^2\cap\mathcal{U}:|\varphi(z)-1|<\delta\})$$
for $\delta>0$ and $\mathcal{U}$ a neighborhood of $(1,1)\in \mathbb{T}^2.$ Take as $\mathcal{U}$ a small bidisc centered at $(1,1)$ with radius $\tilde{\epsilon}=(\sqrt{\epsilon},\sqrt{\epsilon}), \epsilon\in(0,1).$ By Example 4 of \cite{Bergvist} (see page 18-19 of the mentioned article) we know that the \L{}ojasiewicz inequality for the polynomial $p(z)=2-z_1-z_2, z=(z_1,z_2)\in\mathbb{D}^2$ holds for the exponent $q=2.$ As a result, for $\frac{\epsilon}{2}<\delta<1$ one has:
\begin{align}|\varphi(z)+1|=&\left|\frac{2(1-z_1)(1-z_2)}{2-z_1-z_2}\right|&&\\
\leq& \frac{2\epsilon}{C\mathrm{dist}^2(z,(1,1))}&&\\
\leq& \frac{C'\epsilon}{|z_1-1|^2+|z_2-1|^2},   z\in \overline{\mathbb{D}^2}\cap \overline{D^2((1,1)},\tilde{\epsilon}).
\end{align}
Inequality (5.10) lets us conclude that if $\frac{1}{\sqrt{2}}\sqrt{\frac{\epsilon}{2\delta}}<|z_i-1|<\sqrt{\epsilon}$, for $i=1,2$  where $\frac{\epsilon}{2}<\frac{1}{2}<\delta<1,$ then $|\varphi(z)+1|<\delta.$ Following our observations, one receives
\begin{align}
V(\{z\in\mathbb{D}^2\cap\mathcal{U}:|\varphi(z)+1|<\delta\})>&\int_{\{z\in\mathbb{D}^2:\frac{1}{\sqrt{2}}\sqrt{\frac{\epsilon}{2\delta}}<|z_i-1|<\sqrt{\epsilon},i=1,2.\}}dV(z_1,z_2)&&\\
=&C\pi^2\left(\epsilon-\frac{\epsilon}{2\delta}\right)^2&&\\
>&C(\epsilon)\delta^2,
\end{align}
with inequality (5.13) holding for all $\delta$ belonging in a suitable subinterval of $(\frac{1}{2},1)$ (otherwise the set $\Phi^{-1}(S(e,\tilde{\delta}))$ would be empty) for some constant $C(\epsilon)>0$ dependent only $\epsilon>0,$ where $\epsilon$ is not dependent on $\delta.$ 
\end{proof}
\begin{example}
Let
$$\Phi(z_1,z_2)=(\varphi,\varphi)=\left(\frac{3z_1z_2-z_1-z_2}{3-z_1-z_2},\frac{3z_1z_2-z_1-z_2}{3-z_1-z_2}\right), z_1,z_2\in \mathbb{D}.$$
Then, the composition operator $C_{\Phi}:A^2_{\frac{\beta}{2}-2}(\mathbb{D}^2)\to A^2_{\beta}(\mathbb{D}^2), \beta>4,$ is bounded.
\end{example}
\begin{proof} Let $p(z)=3-z_1-z_2, z=(z_1,z_2)\in \overline{\mathbb{D}^2}.$ This polynomial is stable on $\overline{\mathbb{D}^2},$ hence by repeating the estimates appearing in the proof of Theorem \ref{th3}, one receives boundedness of $C_{\Phi}:A^2_{\frac{\beta}{2}-2}(\mathbb{D}^2)\to A^2_{\beta}(\mathbb{D}^2), \beta >4.$
\end{proof}
\section{Conclusions and further discussion}
The study of composition operators in more variables yields a substantial amount challenges and difficulties. Considering symbols that are induced by RIFs seems to be a fruitful subject that connects the theory of composition operators and the theory of RIFs on the bidisc or polydisc. It seems rather difficult to answer Question 1. In order to characterize the Rational Inner Mappings that induce bounded composition operators, a lot more information is needed. Specifically, one has to estimate the difference $|\varphi(z)-\zeta|$ for all $\zeta\in\mathbb{T},$ where $\varphi$ is a RIF on the bidisc or polydisc. Also, another interesting question which seems attractive but difficult to tame is what exactly happens in the case of the polydisc $\mathbb{D}^n, n\ge 3.$
\bibliographystyle{amsplain}

Athanasios Beslikas,\\
Doctoral School of Exact and Natural Studies\\
Institute of Mathematics,\\
Faculty of Mathematics and Computer Science,\\
Jagiellonian University\\ 
\L{}ojasiewicza 6\\
PL30348, Cracow, Poland\\
athanasios.beslikas@doctoral.uj.edu.pl
\end{document}